\documentclass[12pt]{article}
\usepackage[all]{xy}
\usepackage{amssymb,amsmath,amsthm,amscd,amsfonts}
\newtheorem{theorem}{Theorem}[section]
\newtheorem{lemma}[theorem]{Lemma}
\newtheorem{proposition}[theorem]{Proposition}
\newtheorem{corollary}[theorem]{Corollary}
\newtheorem{definition}[theorem]{Definition}
\newtheorem{example}[theorem]{Example}

\newtheorem{remark}[theorem]{Remark}
\numberwithin{equation}{section}

\def\id{{\rm id}}

\date{}
\title{\bf On the Generators in the Category of Actions of Pomonoids on Posets and its Slices}
\author{{\bf Farideh Farsad}\and {\bf Ali Madanshekaf } \\
Department of Mathematics\\Faculty of Mathematics, Statistics and Computer Science\\
Semnan University\\ P. O. Box 35131-19111\\
Semnan\\
Iran\\ emails: faridehfarsad@yahoo.com\\
\qquad amadanshekaf@semnan.ac.ir}
\date{}
\begin{document}
\maketitle
\begin{abstract}
Let $S$ be a pomonoid, in this paper, {\bf Pos}-$S$, the category
of $S$-posets and $S$-poset maps, is considered. First, we
characterize some pomonoids on which all projectives in this
category are generator or free. Then, we study regular injectivity
and weakly regularly $d$-injectivity which lead to some
homological classification results for pomonoids. Among other
things, we get some relationships between regular injectivity in
the slice category {\bf Pos}-$S/B_S$ and generators or cyclic
projectives in {\bf Pos}-$S$.
\end{abstract}
AMS {\it subject classification}: 06F05, 18A32, 18G05, 20M30, 20M50. 
\section{\bf{Introduction and Preliminaries}}
Although there exist many papers which investigate various
properties of generator $S$-acts, among them, there seems to be
known very little on generators $S$-posets. V. Laan investigated
some properties of  generator $S$-posets in \cite{L}. Continuing
this study, in this paper, after some introductory results in
section 1, we attempt in section 2 to collect new results on
generators in {\bf Pos}-$S$ to reply on the questions of
homological classification of pomonoids.

$\mathcal{M}$-injective objects in the slice category
$\mathcal{C}/B$, for any $B$ in $\mathcal{C}$, form the right part
of a weak factorization system that has morphisms of $\mathcal{M}$
as the left part (see~\cite{B.R}). Here, we consider the same case
in the slice category {\bf Pos}-$S/B_S$ of right $S$-poset maps
over $B_S,$ where $B_S$ is an arbitrary $S$-poset.
In section 3, we first, find some conditions for when
all generators are regular $d$-injective or weakly regularly
injective. Finally, we prove that every Emb-injective object in
{\bf Pos}-$S/B_S$ is split epimorphism. By this fact, we received
to some generators and cyclic projectives in {\bf Pos}-$S$ and in
the category of actions of endomorphism pomonoids on posets.

For the rest of this section, we give some preliminaries about
$S$-acts, $S$-posets and slice category which we will need in the
sequel.

Let $S$ be a monoid with identity $1$. Recall that a (right) {\it
$S$-act} is a set $A$ equipped with a map $\mu: A\times S\to A$
called its action, such that, denoting $\mu(a, s)$ by $as$, we
have $a1 = a$ and $a(st)=(as)t$, for all $a\in A$, and $s, t\in
S$. The category of all $S$-acts, with action-preserving ($S$-act)
maps ($f : A\to B$ with $f(as) = f(a)s$, for $s\in S, a\in A$), is
denoted by {\bf Act}-$S$. For instance, take any monoid $S$ and a
non-empty set $A$. Then $A$ becomes a right $S$-act by defining
$as = a$ for all $a\in A$, $s\in S$, we call that $A$ an $S$-act
with trivial action. Clearly $S$ itself is an $S$-act with its
operation as the action. For more information about $S$-acts
see~\cite{K.K.M}.

A monoid  $S$ is said to be a {\it pomonoid}
 if it is also a poset whose partial order
$\leq$ is compatible with the binary operation,
 i.e., $s\leq t, s'\leq t' \ {\rm imply } \ ss'\leq tt'$ (see~\cite{bir}). In this paper
$S$ denotes a pomonoid with an arbitrary order, unless otherwise
stated.

On a monoid $S$ we define the following relations: for every $s,
t\in S$

\begin{enumerate}
\item $s\mathcal{R}t$ if $sS=tS$.
\item $s\mathcal{J}t~~$ if $SsS=StS$.
\item $s\mathcal{D}t$ if there exists $u\in S$ with $sS=uS$ and
$St=Su$.
\end{enumerate}
These relations are called Green's relations on $S$
(see \cite{K.K.M}). Here, we consider these notions for pomonoid
$S$ and supply some suitable results.

Let $S$ be a pomonoid. A {\it {\rm(}right{\rm)} $S$-poset} is a
poset $A$ which is also an $S$-act whose action $\mu : A\times
S\rightarrow A$ is order-preserving, where $A\times S$ is
considered as a poset with componentwise order. The category of
all $S$-posets with action preserving monotone maps is denoted by
{\bf Pos}-$S$. Clearly $S$ itself is an $S$-poset with its
operation as the action. Left $S$-poset can be defined analogously
(see~\cite{F.M}). A left $T$-poset ($_T A$) and right $S$-poset
($A_S$) is called $T$-$S$-biposet (and denoted by $_T A_S$) when
$(ta)s=t(as)$ for every $s\in S, t\in T$ and $a\in A$. We remind
the following results from \cite{L}:

{\rm (i)} For every $A_S$ in {\bf Pos}-$S$, consider the set
End($A_S)={\bf Pos}_S(A, A)$ as a pomonoid with respect to
composition and pointwise order. Also, we define the left
End($A_S$)-action on $A$ by $f\cdot a=f(a)$, for every $a\in A, f\in$
End($A_S$), so that one has $_{\text{End}(A_S)} A_S$.

{\rm (ii)}  The following mappings are pomonoid homomorphisms,
\begin{center}
$\rho: S\to \text{End}(_TA); \quad s\mapsto \rho_s$\\
$\lambda: T\to \text{End}A_S); \quad t\mapsto \lambda_t,$
\end{center}
where $\rho_s:~_{T}A\rightarrow~_TA, \ a\mapsto as$ and $\lambda_t: A_S\to
A_S,\ a\mapsto ta$ are morphisms in $T$-{\bf Pos} and {\bf
Pos}-$S$, respectively.

{\rm (iii)} For every $T$-$S$-biposet $_T A_S$ recall that if
$B\in${\bf Pos}-$S$ then the set ${\bf Pos}_S(B , A$) of all $S$-poset
maps from $B_S$ to $A_S$ is an object in $T$-{\bf Pos} with the action defined by
$$t\cdot f=\lambda_tf$$
for every $f\in{\bf Pos}_S(B , A), t\in T$. Consequently, we have a functor
\begin{center}
{\bf Pos}$_S(-,A):$ {\bf Pos}-$S\to T$-{\bf Pos}
\end{center}
 by taking
$${\bf Pos}_S(-, A)(P)= {\bf Pos}_S(P, A)$$ for every $P\in$ {\bf Pos}-$S$.

An $S$-poset $G_S$ is a generator in the category {\bf Pos}-$S$ if
for any distinct $S$-poset maps $\alpha, \beta: X_S\to Y_S$ there
exists an $S$-poset map $f: G_S\to X_S$ such that $\alpha
f\neq\beta f$.

Given a category $\mathcal{C}$ and an object $B$ of $\mathcal{C}$,
one can construct the {\it slice category} $\mathcal{C}/B$ (read:
$\mathcal{C}$ over $B$): objects of $\mathcal{C}/B$ are morphisms
of $\mathcal{C}$ with codomain $B$, and morphisms in
$\mathcal{C}/B$ from one such object $f: D\to B$
 to another $g: E\to B$ are commutative
triangles in $\mathcal{C}$
$$
\xymatrix{D\ar[rr]^{h}\ar[dr]_{f}& &E\ar[dl]^{g}\\&B&}
$$
i.e, $gh=f$. The composition in $\mathcal{C}/B$ is defined from
the composition in $\mathcal{C},$ in the obvious way. It means
paste triangles side by side.

A poset is said to be {\it complete} if each of its subsets has an
infimum and a supremum, in particular, a complete poset is
bounded, that is, it has the least (bottom) element $\bot$ and the
greatest (top) element $\top$.
\section{\bf{Some Homological Classifications for\\
 Pomonoids by Generators in {\bf Pos}-$S$}}
In this section, we discuss about the generators and projective
generators in {\bf Pos}-$S$ and give some new properties of them.

 As we mentioned in the introduction, generators have already been
 characterized in ~\cite{L} for the category {\bf Pos}-$S$, such as
 the following two first propositions.
\begin{proposition}\label{1}
Cyclic projectives in {\bf Pos}-$S$ are precisely retracts of
$S_S$.
\end{proposition}
\begin{proposition}\label{2}
An $S$-poset $A_S$ is a cyclic projective generator in {\bf
Pos}-$S$ if and only if $A_S\cong eS_S$ for an idempotent $e\in S$
with $e\mathcal{J}1$.
\end{proposition}
It is clear that the previous proposition lead to the following
result:
\begin{proposition}
Let $S$ be a commutative pomonoid. Then all cyclic projective
generators in {\bf Pos}-$S$ are isomorphic to $S_S$.
\end{proposition}
We need  the following proposition from \cite{S}, as a
characterization of cyclic projective $S$-posets.
\begin{proposition}\label{cyclic projective}
Let $A_S$ be an $S$-poset and $a\in A$. Then the following
statement are equivalent:\\
${\rm (i)}~aS_S$ is projective.\\
${\rm (ii)}~aS_S\cong eS_S$ for some idempotent $e\in S.$
\end{proposition}

In~\cite{S} the authors found a decomposition theorem for
projective $S$-posets. In the following, by this fact and
Proposition \ref{2}, we generalize a description of projective
generator $S$-acts to $S$-posets. The proof is almost identical to
that for Proposition 3.18.5 from \cite{K.K.M}.

\begin{theorem}\label{pogenerator}
Every $S$-poset $P_S$ is projective generator if and only if
$P_S=\coprod_{i\in I}P_i$ where $P_i\cong e_iS$ for every $i\in
I$, and at least one $P_j$, $j\in I$ is a generator with
$e_j\mathcal{J} 1$.
\end{theorem}
Notice that for every pomonoid $S$ and an idempotent $e\in S$, the
sub $S$-poset $eS_S$ of $S_S$ is projective according to
Proposition \ref{cyclic projective}, but it is not a generator
because $e\mathcal{J}1$ is not necessary true.

Next, by Theorem~\ref{pogenerator} we get the following result.
\begin{theorem}\label{7}
For any pomonoid $S$ the following statements are equivalent:\\
{\rm (i)} All projective right $S$-posets are generators in {\bf
Pos}-$S$.\\
{\rm (ii)} All cyclic projective  right $S$-posets are generators
in {\bf Pos}-$S$.\\
{\rm (iii)} $e\mathcal{J} 1$ for every idempotent $e\in S$.
\end{theorem}

Recall~\cite{E.M.R} that a left {\it poideal} of a pomonoid $S$ is
a (possibly empty) subset $I$ of $S$ if it is both a monoid left
ideal $(SI\subseteq I )$ and a down set $(a\leq b, b \in I$ imply
$a\in I)$. For example
$$\downarrow eS=\{t\in S~:~\exists s\in S,~t\leq es\}$$
is a cyclic poideal of $S$, for every idempotent $e\in S.$

In the following we shall characterize pomonoids that all their
principal right poideals are generators.
\begin{proposition}\label{cyclic poideal}
Let $S$ be a pomonoid  and $e\in S$ with $e^2=e$. If the cyclic
projective sub $S$-poset $eS_S$ of $S_S$ is a generator in {\bf
Pos}-$S$, then $\downarrow eS$ is also a generator.
\end{proposition}
\begin{proof}
By assumption there exists an $S$-poset epimorphism (is exactly
onto in {\bf Pos}-$S$) $f: eS_S\rightarrow S_S$. Define mapping
$g: \downarrow eS\to S_S$ by $g(x):=f(ex)$ for every $x\in
\downarrow eS$. It is easy to see that $g$ is an $S$-poset map.
Also, for every $s\in S$ there exists $u\in S$ such that
$f(eu)=s$. Then we have
$$g(eu)=f(eeu)=f(eu)=s$$
This means that $g$ is an epimorphism. By Theorem 2.1 in
\cite{L}, we conclude $\downarrow eS$ is a generator, as we
required.
\end{proof}

\begin{lemma}\label{poideal}
Let $S$ be a pomonoid and $z\in S$. If the principal right poideal
$\downarrow zS$ is a generator in {\bf Pos}-$S,$ then there exist $x,
y\in S$ such that $1\leq yx$, and $za\leq zb$, $a, b\in S$ implies
$ya\leq yb$.
\end{lemma}
\begin{proof}
Since $\downarrow zS$ is a generator in {\bf Pos}-$S$, by Theorem
2.1 of \cite{L}, there exists an epimorphism $g : \downarrow
zS\rightarrow S_S$. Hence, there are $u\in \downarrow zS$ and
$t\in S$ such that $u\leq zt$ and $g(u)=1$. Let $y=g(z)$ and
$x=t$. Then $yx=g(z)x=g(zx)$. Since $u\leq zx$, the monotone
property of $g$ implies $g(u)\leq g(zx)$. Consequently,
$1=g(u)\leq g(zx)=yx$. Now, suppose $za\leq zb$, $a, b\in S$. Then
$$ya=g(z)a=g(za)\leq g(zb)=g(z)b=yb.$$
\end{proof}
\begin{proposition}\label{3}
Let $S$ be a pomonoid in which the identity element is the top
element. If all poideals of $S$ are generators then the sub
$S$-poset $eS_S$ of $S_S$ is a generator in  {\bf Pos}-$S$, for
every idempotent $e\in S$.
\end{proposition}
\begin{proof}
By hypothesis, for every idempotent $e\in S$, $\downarrow eS$ is a
generator in {\bf Pos}-$S$. Then by Lemma \ref{poideal}, there
exist $x, y\in S$ such that $1\leq yx$ and $ea\leq eb$, $a, b\in
S$, always implies $ya\leq yb$. In particular, since $e1\leq ee$
then $y\leq ye$. Now, $1\leq yx\leq yex$. Also, evidently, we have
$yex\leq 1$. Consequently, $yex=1$ this means that
$e\mathcal{J}1,$ equivalently $eS_S$ is a projective generator by
Proposition \ref{2}, as we needed.
\end{proof}

\begin{theorem}
Let $S$ be a pomonoid in which the identity element is the top element. The
following statements are equivalent.\\
{\rm (i)} All projective right $S$-posets are generators in {\bf
Pos}-$S$.\\
{\rm (ii)} All  cyclic projective right $S$-posets are generators
in {\bf Pos}-$S$.\\
{\rm (iii)} $e\mathcal{J} 1$ for every idempotent $e\in S$.\\
{\rm (iv)} All principal right poideals of $S$ which are generated
by an idempotent, are generators in {\bf Pos}-$S$.
\end{theorem}
\begin{proof}
${\rm (i)}\Longleftrightarrow{\rm (ii)}\Longleftrightarrow{\rm
(iii)}$: It is Theorem~\ref{7}.\\
${\rm (iii)}\Longrightarrow {\rm (iv)}$. By Proposition~\ref{2} we
get $eS_S$ is a cyclic projective generator. Next,
Proposition~\ref{cyclic poideal}
shows that $\downarrow eS$ is a generator in {\bf Pos}-$S$.\\
${\rm (iv)}\Longrightarrow{\rm (iii)}$. Consider the principal
right poideal $\downarrow eS$ for every idempotent $e\in S$. By a
similar proof of Proposition \ref{3}, the cyclic projective sub
$S$-poset $eS_S$ of $S_S$ is a generator. Using Proposition
\ref{2}, we conclude that $e\mathcal{J}1$.
\end{proof}
By a {\it free $S$-poset on a poset $P$}~\cite{F.M} we mean an
$S$-poset $F$ together with a poset map $\tau: P\rightarrow F$
with the universal property that given any $S$-poset $A$ and a
poset map $f: P\rightarrow A$ there exists a unique $S$-poset map
$\bar{f}: F\rightarrow A$ such that $\bar{f}\circ\tau=f,$ i.e, the
diagram
$$\xymatrix{P\ar[r]^{\tau} \ar[dr]_{f}&F\ar[d]^{\bar{f}}\\ & A}$$
commutes. Also $S$-poset $F$ is given by $F= P\times S$ with
componentwise order and action $(x, s)t=(x, st)$, for $x\in P$ and
$s, t\in S$ (see \cite{F.M}).
\begin{example}
{\rm Let $S$ be a pomonoid generated by the elements $e, k,
k^\prime$ and with discrete order such that $kk^\prime=1$ and
$ek=k^\prime$. Then $eS_S$ is a cyclic projective generator in
{\bf Pos}-$S$. But $eS_S$ is not free (see Lemma \ref{free}
below).}
\end{example}
In the following lemma we shall characterize idempotents of monoid
$S$ which generate free cyclic sub $S$-acts (see Proposition
3.17.17 of \cite{K.K.M}) of $S_S$ generalizes to the category of
$S$-posets. Moreover, we conclude when projective (or cyclic
projective) implies free in {\bf Pos}-$S$.

\begin{lemma}\label{free}
Let $e$ be an idempotent of a pomonoid $S$. Then the sub $S$-poset
$eS_S$ of $S_S$ is a free right $S$-poset if and only if
$e\mathcal{D}1$.
\end{lemma}

\begin{theorem}
For any pomonoid $S$ the following statements are equivalent:\\
{\rm (i)} All projective right $S$-posets are free.\\
{\rm (ii)} All projective generators in {\bf Pos}-$S$ are free.\\
{\rm (iii)} All cyclic projective right S-posets are free.\\
${\rm (iv)}~e\mathcal{D}1$ for every idempotent $e\in S$.
\end{theorem}
\begin{proof}
${\rm (i)}\Longrightarrow {\rm (ii)}$ is trivial.\\
${\rm (ii)}\Longrightarrow {\rm (iii)}$. By Proposition
\ref{cyclic projective}, all cyclic projective $S$-posets are
isomorphism to $eS_S$ for some idempotent $e\in S$. Let
$A=S_S\coprod eS_S$. By Proposition \ref{pogenerator}, $A_S$ is a
projective generator in {\bf Pos}-$S$. By hypothesis
$A_S$ is free which implies that $eS_S$ is free.\\
${\rm (iii)}\Longrightarrow {\rm (i)}$.  By decomposition theorem
in \cite{S}, every projective $S$-poset is isomorphic to a
coproduct of cyclic
projective $S$-posets which are free by assumption. Now we get the result.\\
${\rm (iii)}\Longleftrightarrow {\rm (iv)}$. By characterization
of cyclic
projective $S$-posets in Proposition \ref{cyclic projective} and Lemma \ref{free} we get the equivalence.\\
\end{proof}
\section{\bf{Regular Injectivity in {\bf Pos}-$S$ and {\bf Pos}-$S/B_S$ and Generators}}
Let $\mathcal{C}$ be a category and $\mathcal{M}$ a class of its
morphisms. An object $I$ of  $\mathcal{C}$ is called
$\mathcal{M}$-injective if for each $\mathcal{M}$-morphism $h
:U\to V$ and morphism $u : U\to I$ there exists a morphism $s:
V\to I$ such that $sh = u$. That is, the following diagram is
commutative:
$$\xymatrix{U\ar[d]_{h}\ar[r]^{u} &I\\\ar@{-->}[ur]_{s} V &}$$
In particular, this means that, in the slice category
$\mathcal{C}/B$, $f : X\to B$ is $\mathcal{M}$-injective if, for
any commutative diagram
$$
\xymatrix{U\ar[r]^{u}\ar[d]_{h} &X\ar[d]^f\\V\ar[r]_{v}&B}
$$
with $h\in \mathcal{M}$, there exists an arrow $s: V\to X$ such
that $sh=u$ and $fs=\upsilon$.
$$
\xymatrix{U\ar[r]^{u}\ar[d]_{h}
&X\ar[d]^f\\V\ar@{-->}[ur]_{s}\ar[r]_{v}&B}
$$
Recall that regular monomorphisms (morphisms which are equalizers)
in {\bf Pos}-$S$ (and also in {\bf Pos}-$S/B_S$) are exactly
order-embeddings (see \cite{F.M} and (\cite{FF.M})). In the
following we consider Emb-injectivity in {\bf Pos}-$S$ and {\bf
Pos}-$S/B_S$, where Emb is the class of all order-embeddings of
$S$-posets.
\begin{theorem}\label{6}
All generators in {\bf Pos}-$S$ are Emb-injective if and only if
all $S$-posets are Emb-injective.
\end{theorem}
\begin{proof}
($\Longrightarrow$) Let $A_S$ be an $S$-poset. Consider $A_S\times
S_S$ which is a generator in {\bf Pos}-$S$ and so is
Emb-injective.
Hence $A_S$ is Emb-injective.\\
($\Longleftarrow$) It is clear.
\end{proof}
Note that the class of all embeddings of right poideals into $S_S$
is a subclass of all down-closed embeddings, i.e. all embeddings
$g: B\to C$ with the property that $g(B)$ is down-closed in $C$,
and hence a subclass of all embeddings.

\begin{definition}
{\rm An $S$-poset of $A_S$ is said {\it (principally) weakly
regularly $d$-injective} if it is injective with respect to all
embeddings of (principal) right poideals into $S_S$.}
\end{definition}
\begin{proposition}
If all generators in {\bf Pos}-$S$ are weakly regularly
$d$-injective then all $S$-posets are weakly regularly
$d$-injective.
\end{proposition}
\begin{proof}
Let $A_S$ be an $S$-poset. Consider $A_S\times S_S$ which is a
generator in {\bf Pos}-$S$ and so is weakly regularly
$d$-injective. To show that $A_S$ is weakly regularly
$d$-injective consider the following diagram
$$\xymatrix{I_S \ar@{{>}->}[d]_{i}\ar[r]^{u} &A_S\\ S_S &}$$
where $I$ is a poideal of $S$. Define $S$-poset map $\bar{u}:
I_S\to A_S\times S_S$ by $\bar{u}(s)=(u(s), s)$ for each $s\in S$.
Hence, by assumption, there exists an $S$-poset map $v: S_S\to
A_S\times S_S$ such that $vi=\bar{u}$. By composition with the
projection $\pi_A: A_S\times S_S\to A_S$, we get that $A_S$ is a
weakly regularly $d$-injective.
\end{proof}
Recall that for a pomonoid $S$ an element $s\in S$ is called a
{\it regular element} if there exists $t\in S$ such that $sts=s$.
One calls $S$ a {\it regular pomonoid} if all its elements are
regular.
\begin{theorem}
Let $S$ be a pomonoid whose the identity element is the top element.
Then the following statements are equivalent:\\
{\rm (i)} All $S$-posets are principally weakly regularly $d$-injective.\\
{\rm (ii)} All principal right poideals of $S$ are principally
weakly regularly $d$-injective.\\
{\rm (iii)} All generators in {\bf Pos}-$S$ are principally
weakly regularly $d$-injective.\\
{\rm (iv)} $S$ is a regular pomonoid.
\end{theorem}
\begin{proof}
${\rm (i)}\Longrightarrow {\rm (ii)}$ is trivial.\\
${\rm (i)}\Longleftrightarrow {\rm (iii)}$ Proposition
\ref{6}.\\
${\rm (ii)}\Longrightarrow {\rm (iv)}$ For every $s\in S$,
consider the down closed embedding $i: \downarrow sS\to S_S,
x\mapsto x$. Then it has a left inverse $f$, since $\downarrow sS$
is principally weakly regularly $d$-injective. Now, taking
$f(1)=z$, we have $z\leq st$ for some $t\in S$ and
$$s=f(s)=f(1)s=zs\leq sts.$$
Also, $sts\leq s$, since 1 is the top element of $S$. Therefore
$s$ is a regular element. This means that $S$ is a regular
pomonoid.\\
${\rm (iv)}\Longrightarrow {\rm (i)}$ See Theorem 3.6 in
\cite{S.M}.
\end{proof}
Recall from \cite{E.M.R} that a pomonoid $S$ which has no proper
non-empty left (right) poideal is said to be left (right) {\it
simple}.
\begin{corollary}
If all generators in {\bf Pos}-$S$ are Emb-injective then $S$ is
left simple.
\end{corollary}
\begin{proof}
From hypothesis and Proposition \ref{6}, we conclude that all
complete $S$-posets are Emb-injective. Now Theorem 3.9 in
\cite{E.M.R} allow us to say that $S$ is left simple.
\end{proof}
\begin{proposition}
For any pomonoid $S$ the following statements are equivalent:\\
{\rm (i)}~ All generators in {\bf Pos}-$S$ are complete $S$-posets.\\
{\rm (ii)} All $S$-posets are complete.
\end{proposition}
\begin{proof}
${\rm (i)}\Longrightarrow {\rm (ii)}$. Let $A_S$ be an $S$-poset.
Consider the generator $A_S\times S_S$ which is a complete
$S$-poset by assumption. Now, it is easily seen that $A_S$ is
complete.\\
${\rm (ii)}\Longrightarrow {\rm (i)}$. It is trivial.
\end{proof}

The authors in~\cite{FF.M2}, showed  the following proposition
which is useful in the sequel.

\begin{proposition}\label{Emb-injective object}
Let $S$ be a pomonoid. Suppose $f: A_S\to B_S$ is an Emb-injective
object in {\bf Pos}-$S/B_S$ for any $S$-poset $B_S$. Then $f$ is a
split epimorphism.
\end{proposition}

\begin{corollary} Suppose $f: A_S\to B_S$ is an Emb-injective
object in {\bf Pos}-$S/B_S$. If $A$ is a complete lattice which is
a retract of $A^{(S)}$, then $A_S$ and $B_S$ are Emb-injective in
{\bf Pos}-$S$.
\end{corollary}
\begin{proof}
By hypothesis we conclude that $A_S$ is an Emb-injective $S$-poset
(see \cite{E.M.R}). Also, $f$ is a split epimorphism (Proposition
\ref{Emb-injective object}), consequently $B_S$ as a retract of an
Emb-injective is an Emb-injective $S$-poset.
\end{proof}
\begin{remark}
{\rm There exists a split epimorphism in {\bf Pos}-$S$ which is not
Emb-injective in {\bf Pos}-$S/B_S$. To present an example, take an
arbitrary pomonoid $S$ and $X, B$ are two lattices as shown in the
following
$$
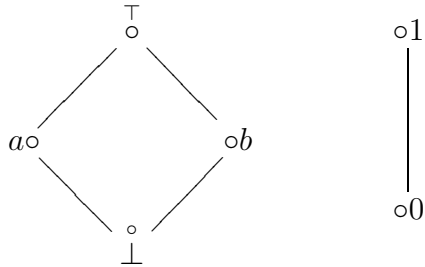
\xymatrix{& \stackrel{\top}{\circ}& \\
a\circ\ar@{-}[ur] & & \circ b\ar@{-}[ul]\\
&\stackrel{\circ}{\bot}\ar@{-}[ul]\ar@{-}[ur] &} \qquad\qquad
\xymatrix{ \circ 1\\ \\\circ 0\ar@{-}[uu]}
$$
Then $X$ is an $S$-poset with the action defined by $\top s=\top$
and $as=bs=\bot s=a$ for all $s\in S$, also we consider $B$ with
the trivial action as an $S$-poset. Define the $S$-poset map $f:
X_S\rightarrow B_S\in \textbf{Pos}-S/B$, by $f(a)=f(b)=f(\bot)=0$
and $f(\top)=1$. It also is a convex map. We show that it is not a
regular injective object in {\bf Pos}-$S/B_S$. Since
$f^{-1}(0)=\{\bot , a , b\}$ is not a complete lattice, the
authors in \cite{FF.M}, showed that it is not Emb-injective in
{\bf Pos}-$S/B_S$.\\
On the other hands, define the $S$-poset map $g: B_S\rightarrow
X_S\in$\textbf{Pos}-$S/B_S,$ by $g(1)=\top, g(0)=g(\bot)$. Then we
have $fg={\text{id}}_B$, so $f$ is split epimorphism. Therefore,
the converse of the above theorem is not true generally.}
\end{remark}
At the rest of this section, we investigate some connections
between Emb-injectivity in {\bf Pos}-$S/B_S$ and generators and
cyclic projectives in {\bf Pos}-$S$.
\begin{theorem}\label{4}
If $f: A_S\to B_S$ is an Emb-injective object in {\bf Pos}-$S/B_S$
and $B_S$ is a generator $S$-poset then $A_S$ is a generator.
Further, $_{\rm{End}(A_S)} A$ is a cyclic projective in
$\rm{End}(A_S)$-{\bf Pos}.
\end{theorem}
\begin{proof}
Since $f: A_S\to B_S$ is Emb-injective so by Proposition
\ref{Emb-injective object}, there exists $g: B_S\to A_S$ in {\bf
Pos}-$S$ such that $fg={\text{id}}_B$. Also, $B_S$ is a generator
in {\bf Pos}-$S$ and $f$ is an epimorphism so $A_S$ is a generator
(see \cite{L}). Now, applying this fact and Theorem 2.2 from
\cite{L}, we get that $_{\text{End}(A_S)} A$ is a cyclic
projective.
\end{proof}
\begin{theorem}
Suppose $f: A_S\to B_S$ is an Emb-injective object in {\bf
Pos}-$S/B_S$ and $A_S$ is a cyclic projective $S$-poset. Then
$B_S$ is a cyclic projective $S$-poset. Moreover, $_{{\rm
End}(B_S)} B$ is a generator in ${\rm End}(B_S)$-{\bf Pos}.
\end{theorem}
\begin{proof}
Since $f: A_S\to B_S$ is Emb-injective so by Proposition
\ref{Emb-injective object}, there exists $g: B_S\to A_S$ such that
$fg={\text{id}}_B$. Also, $A_S$ is a cyclic projective in {\bf
Pos}-$S$ hence by Proposition~\ref{1}, there exist two $S$-poset
maps $\xymatrix{ S_S \ar@<1ex>[r]^{\pi} & A_S \ar@<1ex>[l]^
{\gamma}}$ such that $\pi\gamma={\text{id}}_A$. Then we get
$f\pi\gamma g={\text{id}}_B$, we get $B_S$ is a cyclic projective
by Proposition \ref{1}. Now by Proposition 3.1 from \cite{L}, we
conclude that $_{{\rm End}(B_S)} B$ is a generator.
\end{proof}

\begin{theorem}\label{5}
Suppose $f: A_S\to B_S$ is an Emb-injective object in {\bf Pos}-$S/B_S$. Then\\
${\rm (i)}~{\bf Pos}_S(B_S, A_S)$ is a generator in {\bf
Pos}-$\rm{End}(B_S)$.\\
${\rm (ii)}~{\bf Pos}_S(A_S, B_S)$ is a generator in
$\rm{End}(B_S)$-{\bf Pos}.\\
${\rm (iii)}~{\bf Pos}_S(B_S, A_S)$ is a cyclic projective in
${\rm End}(A_S)$-{\bf Pos}.\\
${\rm (iv)}~{\bf Pos}_S(A_S, B_S)$ is a cyclic projective in {\bf
Pos}-{\rm End}($A_S)$.
\end{theorem}
\begin{proof}
Since $f: A_S\to B_S$ is Emb-injective, in view of Proposition
\ref{Emb-injective object}, there exists $g: B_S\to A_S$ such that
$fg={\text{id}}_B$. Applying the functors {\bf Pos}$_S(B_S, -)$,
{\bf Pos}$_S(-, B_S)$, {\bf Pos}$_S(-, A_S)$ and {\bf Pos}$_S(A_S,
-)$ we get the assertions ${\rm (i)}, {\rm (ii)}, {\rm (iii)}$ and
${\rm (iv)}$ respectively.
\end{proof}

\begin{proposition}
In any of the following cases {\bf Pos}$_S(A_S\times B_S, B_S)$ is
a generator in {\rm End}$(B_S)$-{\bf Pos}, for every $B_S\in${\bf
Pos}-$S$.\\
{\rm (i)} $A_S$ is an Emb-injective $S$-poset.\\
{\rm (ii)} $f: A_S\to B_S$ is an Emb-injective object in {\bf
Pos}-$S/B_S$.
\end{proposition}
\begin{proof}
{\rm (i)} Consider the projection $S$-poset map $\pi_B: A\times
B\to B_S$. The authors in \cite{FF.M} showed that it is an
Emb-injective object in {\bf Pos}-$S/B_S$. Consequently, by
Theorem
\ref{5}(ii), we get the result.\\
{\rm (ii)} By Proposition \ref{Emb-injective object}, there exists
$g: B_S\to A_S$ such that $fg={\text{id}}_B$. Consider the unique
$S$-poset map $\varphi_B : B_S\to A\times B$ to product in such
away that $\pi_B\varphi_B={\text{id}}_B$. Applying the functor
{\bf Pos}$_S(-, B_S)$ we obtain
$$\xymatrix{
{\rm End}(B_S)={\bf Pos}_S(B, B) \ar@<1ex>[r]^{\bar{\pi}_B} & {\bf
Pos}_S(A\times B, B)\ar@<1ex>[l]^{\bar{\varphi_B}}}$$ So we have
$\bar{\varphi}_B\bar{\pi}_B={\text {id}}_{{\text {End}}(B_S)}$.
This means that End($B_S$) is a retract of {\bf Pos}$_S(A\times B,
B)$ as we needed (see Theorem 2.1 of \cite{L}).
\end{proof}

\begin{proposition}
Let $B_S\in${\bf Pos}-$S$ and $_T A_S$ be an $T$-$S$-biposet and
$A\times B$ be a cyclic projective $S$-poset. If $f: A_S\to B_S$
is an Emb-injective object in {\bf Pos}-$S/B_S$ and $\lambda: T\to
{\rm End}(A_S)$ is an isomorphism then $_T A$ is a generator in
$T$-{\bf Pos}.
\end{proposition}
\begin{proof}
Consider the projection $\pi_A: A\times B\to A_S$ and the unique
$S$-poset map $\varphi_A : A_S\to A\times B$ to product in such away
that $\pi_A\varphi_A=\id_A$. On the other hands, $A\times B$ is
cyclic projective $S$-poset so there exist two $S$-poset maps
\xymatrix{ {A\times B}\ar@<1ex>[r]^{\gamma} &
S_S\ar@<1ex>[l]^{\pi}} such that $\pi\gamma=1_{A\times B}$.
Applying the functor ${\bf Pos}_S(-, A_S)$ and the composition
$\pi_A\pi\gamma\varphi_A=1_A$, we obtain
$$\xymatrix{T\cong Pos_S(A, A) \ar@<1ex>[r]^{\bar{\pi}_A} & Pos_S(A\times B,
A)\ar@<1ex>[l]^{\bar{\varphi}_A} \ar@<1ex>[r]^{\bar{\pi}} &
Pos_S(S, A)\cong _T A\ar@<1ex>[l]^{\bar{\gamma}}}$$ Hence, $T$ is
a retract of $_T A$, so $_T A$ is a generator in {\bf Pos}-$S$.
\end{proof}

\end{document}